\documentclass[11pt, reqno]{amsart}
\usepackage{amsmath}
\usepackage[dvips]{graphicx}
\usepackage{amsfonts}
\usepackage{amsopn}
\usepackage{amsthm} 
\usepackage{amssymb}
\usepackage{color}
\usepackage{hyperref}
\usepackage{latexsym}
\usepackage{amsgen}
\usepackage{tikz}

\newtheorem{theorem}{Theorem}
\newtheorem{definition}[theorem]{Definition}

\newtheorem{proposition}[theorem]{Proposition}

\newcommand{\1}{\>}
\newcommand{\2}{\>\>}
\newcommand{\3}{\>\>\>}

\newcommand{\End}{\textbf{end}}

\newcommand{\struc}[1] {\operatornamewithlimits{struc}\left [ #1 \right ]}

\usepackage{tikz}
\usepackage{epstopdf}
\usepackage{epsfig}
\usepackage{graphicx}
\usepackage{subcaption}

\graphicspath{{./figures/}}
\usepackage{tikz}

\usetikzlibrary{shapes,arrows}
\usepackage[ruled,algo2e]{algorithm2e}
\tikzstyle{decision} = [diamond, draw, fill=blue!20, 
    text width=4.5em, text badly centered, node distance=3cm, inner sep=0pt]
\tikzstyle{block} = [rectangle, draw, fill=blue!20, 
    text width=5em, text centered, rounded corners, minimum height=4em]
\tikzstyle{line} = [draw, -latex']
\tikzstyle{cloud} = [draw, ellipse,fill=red!20, node distance=3cm,
    minimum height=2em]
    
\usetikzlibrary{positioning}
\tikzset{main node/.style={circle,fill=blue!20,draw,minimum size=1cm,inner sep=0pt},  }

\begin{document}
\title[unnormalized optimal transport]{Unnormalized Optimal Transport}
\author[Gangbo]{Wilfrid Gangbo}
\author[Li]{Wuchen Li}
\author[Osher]{Stanley Osher}
\author[Puthawala]{Michael Puthawala}
\email{wgangbo@math.ucla.edu}
\email{wcli@math.ucla.edu}
\email{sjo@math.ucla.edu}
\email{mputhawala@math.ucla.edu}
\address{Mathematics department, University of California, Los Angeles}
\newcommand{\vr}{\overrightarrow}
\newcommand{\wt}{\widetilde}
\newcommand{\dd}{\mathcal{\dagger}}
\newcommand{\ts}{\mathsf{T}}
\keywords{Optimal transport; Unnormalized density space; Unnormalized Monge-Amp{\'e}re equation.}
\thanks{The research is supported by AFOSR MURI FA9550-18-1-0502.}
\maketitle
\begin{abstract}
We propose an extension of the computational fluid mechanics approach to the Monge-Kantorovich mass transfer problem, which was developed by Benamou-Brenier in \cite{bb}. Our extension allows optimal transfer of unnormalized and unequal masses. We obtain a one-parameter family of simple modifications of the formulation in \cite{bb}. This leads us to a new Monge-Amp{\'e}re type equation and a new Kantorovich duality formula. These can be solved efficiently by, for example, the Chambolle-Pock primal-dual algorithm \cite{CP}. This solution to the extended mass transfer problem gives us a simple metric for computing the distance between two unnormalized densities. 
The $L_1$ version of this metric was shown in \cite{PO} (which is a precursor of our work here) to have desirable properties. 
\end{abstract}
\section{Introduction}
 Optimal transport (OT) plays important roles in inverse problems \cite{Yang2, Yang1} and machine learning \cite{WGAN, LWL, Wproximal}. It provides a particular distance function, called the Wasserstein metric or Earth Mover's distance, among histograms or density functions \cite{bb, vil2008}. In these traditional settings, it assumes that histograms or densities have the same total mass.  In real applications, we face a situation where the total mass of each histogram is not equal. For example, when comparing two images, their intensities are not the same. This fact prevents us from applying the classical optimal transport.
  
In this paper, we formulate simple and natural extensions of optimal transport in unnormalized density space. In a word, we add a spatial independent source function into the continuity equation and cost functional. There are two benefits of the current approach. On the one hand, the changes of the variational problem are simple. They define a robust $L^p$ Wasserstein metric in unnormalized density space and do not significantly change the computational complexity of the problem. The proposed model allows us to apply classical algorithms, such as the Chambolle-Pock primal-dual method \cite{CP}, to solve it. On the other hand, the proposed problem is natural in that it uses the key Hamilton-Jacobi equation as in the original optimal transport problem. These properties allow us to identify new problems corresponding to the Monge problem and Monge-Amp{\'e}re equation in unnormalized density space.

There have been various extensions of optimal transport for unnormalized or unbalanced densities \cite{B1, Benamou, WF, MC, FG, FG2, Mielke, PR, PR2, RL, D1}. In particular, \cite{WF, WF2, Mielke} propose the Wasserstein-Fisher-Rao or Hellinger--Kantorovich metric\footnote{In the literature, the Wasserstein-Fisher-Rao metric is called unbalanced OT. To distinguish with their approaches, we call our approach unnormalized OT.}. In their studies, a spatially dependent source function is introduced, which is a ratio involving the density in the spatial domain. In addition, \cite{CL} and \cite{Maas} study other spatially dependent source functions. Here we propose a spatially independent source function which keeps the key Hamilton-Jacobi equation as in the normalized case. This property allows us to design a simple algorithm and to derive a reasonable simple unnormalized Monge-Amp{\'e}re equation. 
 
The plan of this paper is as follows. In section \ref{section2}, we propose and study the properties of the unnormalized dynamical optimal transport problem. The unnormalized Monge problem, Monge-Amp{\'e}re equation and Kantorovich formulations are all derived. In section \ref{section4}, we present the algorithms and numerical examples for this proposed metric. 
\section{Unnormalized optimal transport}\label{section2}
In this section, we introduce unnormalized OT problems and show that the proposed unnormalized metric is well defined. We then derive minimization procedures for unnormalized optimal transport. 

Denote $\Omega\subset\mathbb{R}^d$ as a bounded convex domain with area $|\Omega|$. 
Denote the space of normalized densities by
\begin{equation*}
\mathcal{P}(\Omega)=\{\mu\in L^1(\Omega)\colon \mu(x)\geq 0,~\int_ \Omega\mu(x)dx=1\}.
\end{equation*}
Let the space of unnormalized densities be
\begin{equation*}
\mathcal{M}(\Omega)=\{\mu\in L^1(\Omega)\colon \mu(x)\geq 0\}.
\end{equation*}
We note that $\mathcal{P}(\Omega)\subset \mathcal{M}(\Omega)$. We next define the optimal transport cost between $\mu_0,\mu_1\in \mathcal{M}(\Omega)$. 
\begin{definition}[Unnormalized OT]\label{def1}
Define the $L^p$ unnormalized Wasserstein distance $UW_p\colon$ $\mathcal{M}(\Omega)\times \mathcal{M}(\Omega)\rightarrow \mathbb{R}$ by
\begin{subequations}\label{UOTT}
\begin{equation}\label{UOT}
UW_p(\mu_0,\mu_1)^p=\inf_{v,\mu,f}\int_{0}^1\int_\Omega \|v(t,x)\|^{p}\mu(t,x)dxdt+\frac{1}{\alpha} \int_0^1|f(t)|^pdt\cdot |\Omega|
\end{equation}
such that the dynamical constraint, i.e. the unnormalized continuity equation, holds
\begin{equation}\label{UC}
 \partial_t\mu(t,x)+\nabla\cdot(\mu(t,x)v(t,x))= f(t),~\mu(0,x)=\mu_0(x),~\mu(1,x)=\mu_1(x).
\end{equation}
\end{subequations}
Here $\|\cdot\|$ is the Euclidean norm, $\mu_0$, $\mu_1\in\mathcal{M}(\Omega)$, and the infimum is taken over all continuous unnormalized density functions $\mu\colon [0,1]\times \Omega\rightarrow \mathbb{R}$, and Borel vector fields $v\colon [0,1]\times\Omega\rightarrow \mathbb{R}^d$ with zero flux condition $v(t,x)\cdot n(t,x)=0$ on $(0,1)\times \partial \Omega$ with $n(t,x)$ being the normal vector on the boundary of $\Omega$, 
and Borel spatially independent source functions $f\colon [0,1]\rightarrow \mathbb{R}$.  
\end{definition}
The new proposed $L^p$ Wasserstein metric has an attractive physical interpretation. The above optimization problem can be viewed as a variational fluid dynamics problem in Eulerian coordinates. Definition \ref{UOTT} considers the motion, creation and removal of particles. During this process, the total mass is changing dynamically in a uniform manner, controlled by the positive parameter $\alpha$ and a spatially independent function $f(t)$. We remark that the spatial independence of the source function introduces a very important natural property, which we will repeat. It uses the same Hamilton-Jacobi equation as in the classical optimal transport, which allows us to obtain a new Monge problem, Monge-Amp{\'e}re equation and Kantorovich duality problem. In addition, this physical analogy follows approaches in \cite{LiG}. More interestingly, we notice that problem \eqref{UOTT} has essentially the same computational complexity as the classical dynamical optimal transport problem. We will present computational details in section \ref{section4}.

\subsection{$L^1$ unnormalized Wasserstein metric} 
We first study the $L^1$ unnormalized Wasserstein metric. When $p=1$, the problem \eqref{UOT} becomes:
\begin{equation*}
\begin{split}
\textrm{UW}_1(\mu_0,\mu_1)=&\inf_{v(t,x), f(t)}\Big\{\int_{0}^1\int_\Omega \|v(t,x)\|\mu(t,x)dxdt+\frac{1}{\alpha} \int_0^1|f(t)|dt\cdot |\Omega|\colon\\
&\hspace{1.2cm} \partial_t\mu(t,x)+\nabla\cdot(\mu(t,x)v(t,x))=f(t),~\mu(0,x)=\mu_0(x),~\mu(1,x)=\mu_1(x)\Big\}.
\end{split}
\end{equation*}
Denote 
\begin{equation*}
m(x)=\int_0^1 v(t,x)\mu(t,x)dt, 
\end{equation*}
then by Jensen's inequality, the minimizer is obtained by a time independent solution. In other words, $$\int_0^1\int_\Omega\|v(t,x)\|\mu(t,x)dxdt\geq \int_\Omega \|\int_0^1 v(t,x)\mu(t,x)dt\|dx=\int_\Omega \|m(x)\|dx.$$
By integrating the time variable in the constraint, we observe that 
\begin{equation*}
\begin{split}
&\Big\{\int_{0}^1\int_\Omega \|v(t,x)\|\mu(t,x)dxdt+\frac{1}{\alpha} \int_0^1|f(t)|dt\cdot |\Omega|\colon\\
& \partial_t\mu(t,x)+\nabla\cdot(\mu(t,x)v(t,x))=f(t),~\mu(0,x)=\mu_0(x),~\mu(1,x)=\mu_1(x)\Big\}\\
\geq& \Big\{\int_\Omega \|m(x)\|dx+ \frac{1}{\alpha} \int_0^1|f(t)|dt\cdot |\Omega| \colon \mu_1(x)-\mu_0(x)+\int_0^1f(t)dt+\nabla\cdot m(x)=0\Big\}\\
\geq& \Big\{\int_\Omega \|m(x)\|dx+ \frac{1}{\alpha} \Big|\int_0^1f(t)dt\Big|\cdot |\Omega| \colon \mu_1(x)-\mu_0(x)+\int_0^1f(t)dt+\nabla\cdot m(x)=0\Big\}.
\end{split}
\end{equation*}
Denote $c=\int_0^1 f(t)dt$,  by integrating on both time and spatial domain for continuity equation \eqref{UC}, it is clear that $$c=\frac{1}{|\Omega|}\Big(\int_\Omega \mu_0(x)dx-\int_\Omega\mu_1(x)dx\Big).$$ We can show that the minimizer path can be attained in the last inequality, by choosing $\mu(t,x)=(1-t)\mu_0(x)+t\mu_1(x)$.
Thus we derive the following proposition. 
\begin{proposition}\label{prop2}
The $L^1$ unnormalized Wasserstein metric is given by 
\begin{equation*}
\begin{split}
UW_1(\mu_0,\mu_1)=\inf_{m} \Big\{&\int_\Omega \|m(x)\|dx+ \frac{1}{\alpha} \Big|\int_\Omega \mu_0(x)dx-\int_\Omega\mu_1(x)dx\Big| \colon \\
&\mu_1(x)-\mu_0(x)+\frac{1}{|\Omega|}\Big(\int_\Omega \mu_0(x)dx-\int_\Omega\mu_1(x)dx\Big)+\nabla\cdot m(x)=0\Big\}.
\end{split}
\end{equation*}
\end{proposition}
In addition, in one space dimension on the interval $\Omega=[0,1]$, the $L^1$ unnormalized Wasserstein metric has the following explicit solution: 
\begin{equation*}
\begin{split}
UW_1(\mu_0,\mu_1)=&\int_\Omega \Big|\int^x_0\mu_1(y)dy-\int^x_0\mu_0(y)dy- x\int_\Omega(\mu_1(z)-\mu_0(z))dz\Big| dx\\
&+\frac{1}{\alpha} \Big( \Big|\int_\Omega\mu_1(z)dz-\int_\Omega\mu_0(z)dz\Big|\Big).
\end{split}
\end{equation*}

The formulation in proposition \ref{prop2} has been proposed in \cite{PO} for inverse problems. It is one of the prime motivations for this paper. We also note the minimizer satisfies the following form \cite{LiRyuOsherYinGangbo2018_parallel}: 
\begin{equation*}
\left\{
\begin{aligned}
&\frac{m(x)}{\|m(x)\|}=\nabla\Phi(x),\quad \textrm{if $\|m(x)\|\neq 0$}\\
-&\nabla\cdot m(x)=\mu_1(x)-\mu_0(x)+\frac{1}{|\Omega|}\Big(\int_\Omega \mu_0(x)dx-\int_\Omega\mu_1(x)dx\Big).
\end{aligned}\right.
\end{equation*}

\subsection{$L^2$ unnormalized Wasserstein metric}
We next present the result when $p=2$. Similar derivations can also be established for $p\in (1,\infty)$. For simplicity of presentation, we now assume $|\Omega|=1$. 
\begin{proposition}\label{thm2}
The $L^2$ unnormalized Wasserstein metric \eqref{def1} is a well-defined metric function in $\mathcal{M}(\Omega)$. In addition, the minimizer $(v(t,x), \mu(t,x), f(t))$ for problem \eqref{UOTT} satisfies $$v(t,x)=\nabla\Phi(t,x),\quad f(t)=\alpha\int_\Omega\Phi(t,x)dx,$$ and  
\begin{equation}\label{minimizer1}
\left \{
\begin{aligned}
&\partial_t\mu(t,x)+\nabla\cdot (\mu(t,x)\nabla\Phi(t,x))= \alpha\int_\Omega\Phi(t,x)dx \\
&\partial_t\Phi(t,x)+\frac{1}{2}\|\nabla\Phi(t,x)\|^2\leq 0\\
&\mu(0,x)=\mu_0(x),\quad \mu(1,x)=\mu_1(x).
\end{aligned}
\right.
\end{equation}
In particular, if $\mu(t,x)>0$, then 
\begin{equation}\label{HJBM}
\partial_t\Phi(t,x)+\frac{1}{2}\|\nabla\Phi(t,x)\|^2=0.
\end{equation}
\end{proposition}
\noindent\textbf{Remark:}
{\em
We note that equation \eqref{minimizer1} implies 
$$\alpha \int_0^1\int_\Omega \Phi(t,x)dx dt=\int_\Omega \mu_1(y)dy-\int_\Omega \mu_0(y)dy.$$
This means that unlike the classical OT, we are not only solving for the unique $\nabla\Phi$, but also for \textbf{the unique $\Phi$}.
}

\begin{proof}
Denote $m(t,x)=\mu(t,x) v(t,x)$ and 
\begin{equation*}
F(m,\mu)=\begin{cases}
\frac{\|m\|^2}{\mu}& \textrm{if $\mu>0$;}\\
0 & \textrm{if $\mu=0$, $m=0$;}\\
+\infty& \textrm{Otherwise.}
\end{cases}
\end{equation*}
then variational problem \eqref{UOTT} can be reformulated as 
\begin{equation}\label{UOT1}
\begin{split}
\textrm{UW}_2(\mu_0,\mu_1)^2=&\inf_{m,\mu,f}\Big\{\int_{0}^1\int_\Omega F(m(t,x), \mu(t,x))dxdt+\frac{1}{\alpha} \int_0^1|f(t)|^2 dt\colon\\
& \partial_t\mu(t,x)+\nabla\cdot(\mu(t,x)v(t,x))= f(t),~\mu(0,x)=\mu_0(x),~\mu(1,x)=\mu_1(x)\Big\}.
\end{split}
\end{equation}
It is clear that \eqref{UOT1} is the reformulation of \eqref{UOTT}. We first prove that the variational problem \eqref{UOT1} is well defined. In other words, there exists a feasible path for the dynamical constraint. We construct a feasible path $\mu_t$ connecting any $\mu_0$, $\mu_1\in \mathcal{M}(\Omega)$.
The proof is divided into three steps. 

\noindent Step 1. Construct a density path $t\in [0, \frac{1}{3}]$, there exists a feasible path connecting $\mu_0$ and a uniform measure with total mass $\int_\Omega\mu_0dx$. In this case, the density path is a normalized (classical) OT between two densities. We set $f(t)=0$ when $t=[0, 1/3]$, there always exists such a path.

\noindent Step 2. Construct a density path $t\in [\frac{1}{3}, \frac{2}{3}]$, there exists a feasible path connecting a uniform measure with total mass $\int_\Omega\mu_0dx$ and a uniform measure with total mass $\int_\Omega\mu_1dx$. In this case, we let the transport flux $m(t,x)=0$, and choose $f(t)=3 (\int_\Omega \mu^1(x)dx- \int_\Omega \mu^0(x)dx)$. 

\noindent Step 3. Construct a density path $t\in [\frac{2}{3}, 1]$, there exists a feasible path connecting a uniform measure with total mass $\int_\Omega\mu_1dx$ and $\mu_1$. In this case, we set $f(t)=0$. Following the classical OT, we find a feasible path. 

Combining steps 1,2,3, the proposed path is feasible with finite cost functional. We next show that the problem has a minimizer. Since the constraint set is not empty, then it is classical to show the cost functional $F(m,\mu)+\frac{1}{\alpha}f(t)^2$ is convex and is lower semicontinuous, while the constraint is linear. So the variational problem \eqref{minimizer1} has a minimizer. 

We next apply a Lagrange multiplier to find the minimizer. Denote $\Phi(t,x)$ as the multiplier with
\begin{equation*}
\begin{split}
\mathcal{L}(m, \mu,\Phi)=&\int_0^1\int_\Omega \frac{\|m(t,x)\|^2}{2\mu(t,x)}+\Phi(t,x)\Big(\partial_t\mu(t,x)+\nabla\cdot m(t,x)-f(t)\Big)dxdt+\frac{1}{2\alpha}\int_0^1 f(t)^2dt.
\end{split}
\end{equation*}
Assuming $\delta_m \mathcal{L}=0$, $\delta_\mu\mathcal{L}\leq 0$, $\delta_f\mathcal{L}=0$, we derive the property of minimizer as follows:
\begin{equation*}
\left\{\begin{split}
&\frac{m(t,x)}{\mu(t,x)}=\nabla\Phi(t,x)\\
&-\frac{m(t,x)^2}{2\mu(t,x)^2}-\partial_t\Phi(t,x)\leq 0\\
&f(t)=\alpha\int_\Omega \Phi(t,x)dx.
\end{split}
\right.
\end{equation*}
Here if $\mu>0$, we obtain $\delta_\mu\mathcal{L}=0$, which gives equality in the second formula of the above system. Using the fact $\frac{m(t,x)}{\mu(t,x)}=\nabla\Phi(t,x)$, we prove the result. In this case, the non-negativity, symmetry, triangle inequality of the metric follow directly from the definition. 
\end{proof}

We next derive our new Monge problem for unnormalized OT. This approach uses the Lagrange coordinates arising in problem \eqref{UOTT}. 
\begin{proposition}[Unnormalized Monge problem]
\begin{subequations}
\begin{equation}\label{UM}
\begin{split}
\textrm{UW}_2(\mu_0,\mu_1)^2=&\inf_{M, f(t)}~\int_\Omega \|M(x)-x\|^2 \mu_0(x)dx+\alpha\int_0^1 f(t)^2 dt\\
&+\int_0^1 \int_0^t f(s) \int_\Omega\|M(x)-x\|^2\textrm{Det}\Big(s\nabla M(x)+(1-s)\mathbb{I}\Big)ds dt dx, 
\end{split}
\end{equation}
where the infimum is among all one to one, invertible mapping functions $M\colon \Omega\rightarrow \Omega$ and a source function $f\colon \Omega\rightarrow \mathbb{R}$, such that the unnormalized push forward relation holds 
\begin{equation}\label{UMC}
\mu(1,M(x))\textrm{Det}(\nabla M(x))=\mu(0, x)+\int_0^1 f(t) \textrm{Det}\Big(t\nabla M(x)+(1-t) \mathbb{I}\Big)dt.
\end{equation}
\end{subequations}
\end{proposition}

\begin{proof}
We now derive the Lagrange formulation of the unnormalized OT \eqref{UOTT}. Consider any mapping function $X_t(x)$ with vector field $v(t, X_t(x))$, i.e. 
\begin{equation*}
\frac{d}{dt}X_t(x)=v(t, X_t(x)),\quad X_0(x)=x.
\end{equation*}
Then 
\begin{equation}\label{new}
\begin{split}
\int_{0}^1\int_\Omega \|v(t,x)\|^2\mu(t,x)dxdt=&\int_{0}^1\int_\Omega \|v(t,X_t(x))\|^2\mu(t,X_t(x))dX_t(x)dt\\
=&\int_0^1 \int_\Omega \|\frac{d}{dt}X_t(x)\|^2\mu(t, X_t(x))\textrm{Det}\Big(\nabla X_t(x)\Big)dx dt.
\end{split}
\end{equation}
We next derive the differential equation for $J(t,x):=\mu(t, X_t(x))\textrm{Det}\Big(\nabla X_t(x)\Big)$. Later on, we use the notation $J(t)=J(t,x)$ and $\frac{d}{dt}J(t)=\frac{\partial}{\partial t}J(t,x)$.
Since
\begin{equation*}
\begin{split}
\frac{d}{d t}J(t,x)=&\frac{d}{d t}\Big\{\mu(t, X_t(x))\textrm{Det}\Big(\nabla X_t(x)\Big)\Big\}\\
=&\partial_t\mu(t, X_t(x)) \textrm{Det}\Big(\nabla X_t(x)\Big) +\nabla_X \mu(t, X_t(x))\frac{d}{dt}X_t(x)\textrm{Det}\Big(\nabla X_t(x)\Big)\\
&+\mu(t, X_t(x))\partial_t \textrm{Det}\Big(\nabla X_t(x)\Big)\\
=&\Big\{\partial_t\mu+ \nabla\mu\cdot v+\nabla\cdot v \mu\Big\}(t, X_t(x))\textrm{Det}(\nabla X_t(x))\\
=&\Big\{\partial_t\mu+\nabla\cdot(\mu v)\Big\}(t, X_t(x))\textrm{Det}(\nabla X_t(x))\\
=& f(t)\textrm{Det}(\nabla X_t(x)),
\end{split}
\end{equation*}
where the third equality is derived by the Jacobi identity, i.e.
\begin{equation*}
\partial_t \textrm{Det}\Big(\nabla X_t(x)\Big)=\nabla\cdot v(t, X_t(x))\textrm{Det}\Big(\nabla X_t(x)\Big),
\end{equation*}
 and the last equality holds following our proposed continuity equation with spatial independent source function \eqref{UC}. 

Notice \begin{equation*}
J(t)=J(0)+\int_0^t \frac{d}{ds}J(s)ds.
\end{equation*}
Since $X_0(x)=x$ and $\nabla X_0(x)=\mathbb{I}$, then $J(0)=\mu(0,x)$ and 
\begin{equation*}
\mu(t, X_t(x))\textrm{Det}\Big(\nabla X_t(x)\Big)=\mu(0,x)+\int_0^t f(s)\textrm{Det}\Big(\nabla X_s(x)\Big)ds.
\end{equation*}
Since the minimizer in Eulerian coordinates satisfies the Hamilton-Jacobi equation in \eqref{HJBM}:
\begin{equation*}
\partial_t\Phi(t,x)+\frac{1}{2}\|\nabla\Phi(t,x)\|^2=0,
\end{equation*}
and $\frac{d}{dt}X_t(x)=\nabla\Phi(t, X_t(x))$, then we naturally have 
$\frac{d^2}{dt^2}X_t(x)=0$. This implies
\begin{equation*}
\frac{d}{dt}X_t(x)=v(t, X_t(x))=M(x)-x, 
\end{equation*}
thus $X_t(x)=(1-t)x+tM(x)$ and $\textrm{Det}\Big(\nabla X_t(x)\Big)=\textrm{Det}\Big((1-t)\mathbb{I}+t\nabla M(x)\Big)$.

Substituting all the above relations into \eqref{new}: 
\begin{equation*}
\begin{split}
\eqref{new}=&\int_0^1 \int_\Omega \|\frac{d}{dt}X_t(x)\|^2J(t) dx dt\\ 
=&\int_0^1 \int_{\Omega} \|M(x)-x\|^2 \Big(J(0)+\int_0^t \frac{d}{ds}J(s)ds\Big)dxdt\\
=&\int_0^1 \int_{\Omega} \|M(x)-x\|^2 J(0) dxdt+\int_0^1 \int_{\Omega} \|M(x)-x\|^2 \int_0^t \frac{d}{ds}J(s)dsdxdt\\
=&\int_0^1 \int_\Omega \|M(x)-x\|^2\mu(0,x)dx dt+ \int_0^1 \int_\Omega \|M(x)-x\|^2 \int_0^t f(s) \textrm{Det}(\nabla X_s(x))dsdxdt\\
=&\int_\Omega \|M(x)-x\|^2\mu(0,x)dx+ \int_0^1 \int_0^t \int_\Omega \|M(x)-x\|^2 f(s) \textrm{Det}\Big((1-s)\mathbb{I}+s\nabla M(x)\Big)dsdxdt.
\end{split}
\end{equation*}
Thus we prove the results. 
\end{proof}
We next find the relation between the spatial independent source function $f(t)$ and the mapping function $M(x)$. For simplicity of presentation, we assume periodic boundary conditions on $\Omega$. 
\begin{proposition}[Unnormalized Monge-Amp{\'e}re equation]
The optimal mapping function $M(x)=\nabla \Psi(x)$ satisfies the following unnormalized Monge-Amp{\'e}re equation
\begin{equation*}
\begin{split}
&\mu(1,\nabla\Psi(x))\textrm{Det}(\nabla^2\Psi(x))-\mu(0, x)\\
=&\alpha\int_0^1 \textrm{Det}\Big(t\nabla^2\Psi(x)+(1-t) \mathbb{I}\Big)\int_\Omega  \Big(\Psi(y)-\frac{\|y\|^2}{2}+\frac{t\|\nabla \Psi(y)-y\|^2}{2}\Big) \textrm{Det}\Big(t\nabla^2\Psi(y)+(1-t) \mathbb{I}\Big)dy dt.
\end{split}
\end{equation*}
\end{proposition}

\begin{proof}
Let us rewrite the minimizer \eqref{minimizer1} into a time independent formulation. From the Hopf-Lax formula for the Hamilton-Jacobi equation, 
\begin{equation*}
\Phi(1, M(x))=\Phi(0, x)+\frac{\|M(x)-x\|^2}{2}.
\end{equation*}
Thus $\nabla \Phi(0,x)+x-M(x)=0$. We further denote $\Psi(x)=\Phi(0,x)+\frac{\|x\|^2}{2}$, then $M(x)=\nabla\Psi(x)$. From $X_t(x)=(1-t)x+tM(x)$, then
\begin{equation*}
\begin{split}
\Phi(t, X_t(x))=&\Phi(0,x)+\frac{\|X_t(x)-x\|^2}{2t}\\
=&\Phi(0,x)+\frac{t\|M(x)-x\|^2}{2}\\
=&\Psi(x)-\frac{\|x\|^2}{2}+\frac{t\|\nabla \Psi(x)-x\|^2}{2}
\end{split}
\end{equation*}
and \begin{equation*}
\nabla X_t(x)=(1-t)\mathbb{I}+t\nabla^2\Psi(x).
\end{equation*}
From \eqref{minimizer1} and the above two formulas, then 
\begin{equation*}
\begin{split}
f(t)=&\alpha \int_\Omega\Phi(t,x)dx= \alpha \int_\Omega \Phi(t,X_t(x))dX_t(x)\\
=&\alpha \int_\Omega \Phi(t,X_t(x))\textrm{Det}\Big(\nabla X_t(x)\Big)dx\\
=&\alpha \int_\Omega \Big\{\Psi(x)-\frac{\|x\|^2}{2}+\frac{t\|\nabla \Psi(x)-x\|^2}{2}\Big\}\textrm{Det}\Big((1-t)\mathbb{I}+t\nabla^2\Psi(x)\Big)  dx.
\end{split}
\end{equation*}
Substituting $f(t)'s$ formula and $M(x)=\nabla\Psi(x)$ into \eqref{UMC}, we derive the result. 
\end{proof}

We now present the Kantorovich duality formulation of the problem \eqref{UOTT}. 
\begin{proposition}[Unnormalized Kantorovich formulation]
\begin{subequations}
\begin{equation*}\label{UM}
\begin{split}
\frac{1}{2}\textrm{UW}_2(\mu_0,\mu_1)^2
=&\sup_{\Phi}~\Big\{\int_{\Omega}\Phi(1,x)\mu(1,x)dx-\int_\Omega \Phi(0,x)\mu(0,x)dx-\frac{\alpha}{2}\int_0^1\Big(\int_\Omega \Phi(t,x)dx\Big)^2dt\Big\}
\end{split}
\end{equation*}
where the supremum is taken among all $\Phi\colon [0,1]\rightarrow\Omega$ satisfying 
\begin{equation*}
\partial_t\Phi(t,x)+\frac{1}{2}\|\nabla\Phi(t,x)\|^2\leq 0.
\end{equation*}
\end{subequations}
\end{proposition}
\begin{proof}
As in \cite{Gangbo1994, Gangbo1996}, we derive the duality formula by integration by parts as follows. Notice the fact that 
\begin{equation*}
\begin{split}
&\frac{1}{2}\textrm{UW}_2(\mu_0, \mu_1)^2\\
=&\inf_{m, \mu, f} \Big\{\int_0^1 \int_\Omega \frac{m(t,x)^2}{2\mu(t,x)}dxdt+\frac{1}{2\alpha} \int_0^1 f(t)^2dt\colon  \partial_t\mu+\nabla\cdot m=0,~\mu(0,x)=\mu_0(x),~\mu(1,x)=\mu_1(x)\Big\}\\
=&\inf_{m, \mu, f}\sup_{\Phi}\Big\{\int_0^1\int_\Omega \frac{m(t,x)^2}{2\mu(t,x)}+\frac{1}{2\alpha} f(t)^2+\Phi(t,x)\Big(\partial_t\mu(t,x)+\nabla\cdot m(t,x)-f(t)\Big)dxdt\Big\}\\
\geq &\sup_{\Phi}\inf_{m, \mu, f}\Big\{\int_0^1\int_\Omega \frac{m(t,x)^2}{2\mu(t,x)}+\frac{1}{2\alpha} f(t)^2+\Phi(t,x)\Big(\partial_t\mu(t,x)+\nabla\cdot m(t,x)-f(t)\Big)dxdt\Big\}\\
=&\sup_{\Phi}\inf_{m, \mu, f}\Big\{\int_0^1\int_\Omega \frac{m(t,x)^2}{2\mu(t,x)}-\nabla\Phi(t,x) \cdot m(t,x)+\frac{1}{2\alpha} f(t)^2+\Phi(t,x)\cdot \Big(\partial_t\mu(t,x)-f(t)\Big)dxdt\Big\}\\
=&\sup_{\Phi}\inf_{m, \mu, f}\Big\{\int_0^1\int_\Omega \frac{1}{2}\Big(\frac{m(t,x)}{\mu(t,x)}-\nabla\Phi(t,x)\Big)^2\mu(t,x)-\frac{1}{2}\|\nabla\Phi(t,x)\|^2\mu(t,x) dxdt\\
&\hspace{1.5cm}+\int_\Omega \Big(\Phi(1,x)\mu(1,x) -\Phi(0,x)\mu(0,x)\Big) dx\\
&\hspace{1.5cm}+\int_0^1\int_\Omega \Big(-\mu(t,x)\partial_t\Phi(t,x)+\frac{1}{2\alpha}f(t)^2-\Phi(t,x)f(t)\Big)dxdt\Big\}\\
=&\sup_{\Phi}\Big\{\int_\Omega \Big(\Phi(1,x)\mu(1,x) -\Phi(0,x)\mu(0,x)\Big) dx\\
&\qquad+\inf_{\mu} \int_0^1\int_\Omega -\mu(t,x)\Big(\partial_t\Phi(t,x)+\frac{1}{2}\|\nabla\Phi(t,x)\|^2\Big)dxdt\\
&\qquad+\inf_{f}\int_0^1\Big(\int_\Omega\frac{1}{2\alpha}f(t)^2-\Phi(t,x)f(t)\Big)dxdt\Big\}\\
=&\sup_{\Phi}\Big\{\int_\Omega \Big(\Phi(1,x)\mu(1,x) -\Phi(0,x)\mu(0,x)\Big) dx-\frac{1}{2\alpha} \int_0^1\Big(\int_\Omega \Phi(t,x)dx\Big)^2dt\\
&\qquad+\inf_{\mu}\Big\{ -\int_0^1\int_\Omega \mu(t,x)\Big(\partial_t\Phi(t,x)+\frac{1}{2}\|\nabla\Phi(t,x)\|^2\Big)dxdt\Big\}\\
&\qquad+\frac{1}{2\alpha}\inf_{f}\int_0^1\Big(f(t)-\alpha \int_\Omega \Phi(t,x) dx\Big)^2dt\Big\}\\
=&\sup_{\Phi}\Big\{\int_\Omega \Big(\Phi(1,x)\mu(1,x) -\Phi(0,x)\mu(0,x)\Big) dx-\frac{1}{2\alpha} \int_0^1\Big(\int_\Omega \Phi(t,x)dx\Big)^2dt\colon \\
&\qquad\qquad\partial_t\Phi(t,x)+\frac{1}{2}\|\nabla \Phi(t,x)\|^2\leq 0 \Big\}.
\end{split}
\end{equation*}
We have shown that the minimizer over $m$ is obtained at $\frac{m}{\mu}=\nabla \Phi$, and $f(t)=\alpha \int_\Omega \Phi(t,x)dx$. The last equality holds because $\mu(t,x)\geq 0$, thus $\partial_t\Phi(t,x)+\frac{1}{2}\|\nabla\Phi(t,x)\|^2\leq 0$. 

We next show that the primal-dual gap is zero. From proposition \ref{thm2}, the minimizer $(\mu, \Phi)$ satisfies \eqref{minimizer1}. Thus  
\begin{equation*}
\begin{split}
&\int_0^1\int_\Omega \frac{m(t,x)^2}{2\mu(t,x)}dxdt+\frac{1}{2\alpha} \int_0^1 f(t)^2dt\\
=&\int_0^1\int_\Omega \frac{1}{2}\|\nabla\Phi(t,x)\|^2\mu(t,x)dxdt+\frac{\alpha}{2}\int_0^1 \Big(\int_\Omega\Phi(t,x)dx\Big)^2dt\\
=&\int_0^1\int_\Omega \Big(- \frac{1}{2} \|\nabla\Phi(t,x)\|^2\mu(t,x)+\|\nabla\Phi(t,x)\|^2\mu(t,x)\Big) dxdt+\frac{\alpha}{2}\int_0^1 \Big(\int_\Omega\Phi(t,x)dx\Big)^2dt\\
=&\int_0^1\int_\Omega \partial_t\Phi(t,x)\mu(t,x)+\Phi(t,x) \Big(-\nabla\cdot(\mu(t,x)\nabla\Phi(t,x))\Big)dxdt +\frac{\alpha}{2}\int_0^1 \Big(\int_\Omega\Phi(t,x)dx\Big)^2dt\\
=& \int_\Omega \Phi(1,x)\mu(1,x)dx-\int_\Omega \Phi(0,x)\mu(0,x)dx \\
&-\int_0^1\int_\Omega \Phi(t,x) \Big(\partial_t\mu(t,x)+\nabla\cdot (\mu(t,x)\nabla\Phi(t,x))\Big) dx dt+\frac{\alpha}{2}\int_0^1 \Big(\int_\Omega\Phi(t,x)dx\Big)^2dt\\
=& \int_\Omega \Phi(1,x)\mu(1,x)dx-\int_\Omega \Phi(0,x)\mu(0,x)dx \\
&-\int_0^1\int_\Omega \Phi(t,x) f(t) dx dt+\frac{\alpha}{2}\int_0^1 \Big(\int_\Omega\Phi(t,x)dx\Big)^2dt\\
=& \int_\Omega \Phi(1,x)\mu(1,x)dx-\int_\Omega \Phi(0,x)\mu(0,x)dx+ (-\alpha+\frac{\alpha}{2})\int_0^1 \Big(\int_\Omega\Phi(t,x)dx\Big)^2dt\\
=& \int_\Omega \Phi(1,x)\mu(1,x)dx-\int_\Omega \Phi(0,x)\mu(0,x)dx-\frac{\alpha}{2}\int_0^1 \Big(\int_\Omega\Phi(t,x)dx\Big)^2dt.
\end{split}
\end{equation*}
This concludes the proof.
\end{proof}
\section{The numerical method}\label{section4}
In this section, we propose to apply a primal-dual algorithm to solve unnormalized OT numerically. We then provide several numerical examples to demonstrate the effectiveness of this procedure.   
\subsection{Algorithm}
We present a primal-dual algorithm for problem \eqref{UOTT}. In particular, our method is based on its reformulation \eqref{UOT1}, named the minimal flux problem.  
Define the Lagrangian of \eqref{UOT1}: 
\begin{equation*}
\begin{split}
\mathcal{L}(m,\mu, f, \Phi)=&\int_0^1\int_\Omega \frac{\|m(t,x)\|^2}{2\mu(t,x)}dtdx+\frac{1}{2\alpha}\int_0^1 f(t)^2dt\\
&+\int_0^1\int_\Omega\Phi(t,x)\Big(\partial_t\mu(t,x)+\nabla\cdot m(t,x)-f(t)\Big)dxdt,
\end{split}
\end{equation*}
where $\Phi(t,x)$ is the Lagrange multiplier of the unnormalized continuity equation \eqref{UC}.  

Convex analysis shows that
$(m^*(t,x), \mu^*(t,x), f^*(t))$ is a solution to \eqref{UOT1} if and only if there is a $\Phi^*$ such that 
$(m^*, \Phi^*)$ is a saddle point of $\mathcal{L}(m, \mu, f, \Phi)$.
In other words, we can compute minimization \eqref{UOT1} by solving the following minimax problem
\begin{equation*}
\inf_{m, \mu, f} \sup_{\Phi}~\mathcal{L}(m,\mu, f, \Phi),
\end{equation*}
It is clear that $\mathcal{L}$ is convex in $m$, $\mu$, $f$ and concave in $\Phi$, and the interaction term is a linear operator. This property allows us to apply the Chambolle-Pock first order primal-dual algorithm \cite{CP}, which gives the update as follows.
\begin{equation}\label{update}
\left\{
\begin{split}
m^{k+1}(t,x)=&\arg\inf_{m} ~\mathcal{L}(m,\mu^k, f^k, \Phi^k)+\frac{1}{2\tau_1}\int_0^1\int_\Omega\|m(t,x)-m^k(t,x)\|^2dxdt\\
\mu^{k+1}(t,x)=&\arg\inf_{\mu} ~\mathcal{L}(m^k,\mu, f^k, \Phi^k)+\frac{1}{2\tau_1}\int_0^1\int_\Omega\|\mu(t, x)-\mu^k(t,x)\|^2dxdt\\
f^{k+1}(t)=&\arg\inf_{f}~\mathcal{L}(m^k,\mu^k, f, \Phi^k)+\frac{1}{2\tau_1}\int_0^1\|f(t)-f^k(t)\|^2 dt\\
\tilde\Phi^{k+1}(t,x)=&\arg\sup_{\Phi}~\mathcal{L}(\tilde m, \tilde \mu, \tilde f, \Phi)-\frac{1}{2\tau_2}\int_0^1\int_\Omega\|\Phi(t,x)-\Phi^k(t,x)\|^2dxdt\\
(\tilde m, \tilde \mu, \tilde f)=&2 (m^{k+1}, \mu^{k+1},  f^{k+1})-(m^{k}, \mu^{k},  f^{k})
\end{split}\right.
\end{equation}
where $\tau_1$, $\tau_2$ are given step sizes for primal, dual variables. These steps can be interpreted as a gradient descent in the primal variable $(m, \mu, f)$ and a gradient ascent in the dual variable $\Phi$.

It turns out that the optimizations in above update \eqref{update} have explicit formulas. The first line becomes 
\begin{equation*}
\begin{split}
m^{k+1}(t,x)=&\arg\inf_{m} ~ \Big\{\frac{\|m(t,x)\|^2}{2\mu^k(t,x)}-m(t,x)\cdot \nabla\Phi(t,x)+\frac{1}{2\tau_1}\|m(t,x)-m^k(t,x)\|^2\Big\}\\
=&\frac{\mu^k(t,x)}{\mu^k(t,x)+\tau_1}\Big(\tau_1\nabla\Phi(t,x)+m^k(t,x)\Big).
\end{split}
\end{equation*}
The second line of \eqref{update} simplifies to
\begin{equation*}
\begin{split}
\mu^{k+1}(t,x)=&\arg\inf_{\mu} ~\frac{\|m^k(t,x)\|^2}{2\mu(t,x)}-\partial_t\Phi(t,x)\cdot\mu(t,x) +\frac{1}{2\tau_1}|\mu(t,x)-\mu^k(t,x)|^2.
\end{split}
\end{equation*}
The above problem has an analytical solution by solving a cubic equation. The third line of \eqref{update} gives 
\begin{equation*}
\begin{split}
f^{k+1}(t)=&\arg\inf_{f} ~\Big\{\frac{1}{2\alpha}f(t)^2-f(t)\int_\Omega\Phi(t,x)dx +\frac{1}{2\tau_1}\|f(t)-f^k(t)\|^2\Big\} \\
=&\frac{\alpha}{\alpha+\tau_1}\Big(\tau_1\int_\Omega\Phi(t,x)dx+f^k(t)\Big).
\end{split}
\end{equation*}
The fourth line of \eqref{update} gives
\begin{equation*}
\begin{split}
\Phi^{k+1}(t,x)=& \arg\sup_\Phi\Big\{\Phi(t,x)\cdot(\partial_t\tilde\mu(t,x)+\nabla\cdot \tilde m(t,x) -\tilde f(t)) -\frac{1}{2\tau_2}\|\Phi(t,x)-\Phi^k(t,x)\|^2\Big\}\\
=&\Phi^k(t,x)+\tau_2\Big(\partial_t\tilde\mu^{k+1}(t,x)+\nabla\cdot \tilde m(t,x)-\tilde f(t)\Big).
\end{split}
\end{equation*}

Combining all above formulas, we are now ready to state the algorithm. 
\begin{tabbing}
aaaaa\= aaa \=aaa\=aaa\=aaa\=aaa=aaa\kill  
   \rule{\linewidth}{0.8pt}\\
   \noindent{\large\bf Algorithm: Primal-Dual method for Unnormalized OT}\\
  \1 \textbf{Input}: Unnormalized densities $\mu_0$, $\mu_1$; \\
    \3 Initial guess of $m^0$, $\mu^0$, $\Phi^0$, $f^0$, step size $\tau_1$, $\tau_2$.\\
  \1 \textbf{Output}: Minimizer $\mu(t,x)$; Dual variable $\Phi(t,x)$; Value $\textrm{UW}_2(\mu_0, \mu_1)$.\\
   \rule{\linewidth}{0.5pt}\\
1.  \1 \textbf{For} $k=1, 2, \cdots$ Iterate until convergence\\
2.  \2  $m^{k+1}(t,x)=\frac{\mu^k(t,x)}{\mu^k(t,x)+\tau_1}\Big(\tau_1\nabla\Phi(t,x)+m^k(t,x)\Big)$;\\
3.  \2  Solve $\mu^{k+1}(t,x)=\arg\inf_{\mu} ~\frac{\|m^k(t,x)\|^2}{2\mu(t,x)}-\partial_t\Phi(t,x)\cdot\mu(t,x)+\frac{1}{2\tau_1}|\mu(t,x)-\mu^k(t,x)|^2;$\\
4.  \2 $f^{k+1}(t)=\frac{\alpha}{\alpha+\tau_1}\Big(\tau_1\int_\Omega\Phi(t,x)dx+f^k(t)\Big)$;\\
5. \2 $\Phi^{k+1}(t,x)=\Phi^k(t,x)+\tau_2\Big(\partial_t\tilde\mu^{k+1}(t,x)+\nabla\cdot \tilde m(t,x)-\tilde f(t)\Big)$;\\
6. \2 $(\tilde m, \tilde \mu, \tilde f)=2 (m^{k+1}, \mu^{k+1},  f^{k+1})-(m^{k}, \mu^{k},  f^{k})$;\\
7. \1 \End\\
   \rule{\linewidth}{0.8pt}
\end{tabbing}
\subsection{Numerical Grid}

To apply the algorithm, we first define our numerical grid. For simplicity we consider the case where the space of interest is $\Omega = [0, 1]^d$ and time $\mathcal{T} = [0,1]$. Further, for the following explanations we consider the problem when $d = 2$, however, our grid construction can be constructed on any dimension by extending it in the obvious way. We will use the same symbol to represent both the continuous $u, m, \Phi, f$ and their respective discretized counterparts, as the difference between the two should be clear from context alone.

Let $n_t, n_x$, and $n_y$ be given then notate $\Delta t = \frac{1}{n_t}$, $\Delta x = \frac{1}{n_x}$, and $\Delta y = \frac{1}{n_y}$. Using this notation we define the following sets:
\begin{align*}
     \Omega_{(i,j)} &= [i\Delta x, (i+1)\Delta x] \times [j \Delta y, (j + 1)\Delta y]\\
    \mathcal{T}_{(k)} &= [k \Delta t, (k + 1)\Delta t]\\
    \Omega_{(i - 1/2,j)} &= [(i - 1/2)\Delta x, (i + 1/2)\Delta x] \times [j \Delta y, (j + 1)\Delta y]  \text{ for } i = 0,\dots, n_x\\
    \Omega_{(i,j - 1/2)} &= [i\Delta x, (i + 1)\Delta x] \times [(j - 1/2) \Delta y, (j + 1/2)\Delta y] \text{ for } j = 0, \dots, n_y
\end{align*}
where $i = 0,\dots, n_x - 1$, $j = 0, \dots, n_y - 1$, and $k = 0\dots, n_t - 1$ unless otherwise specified.

For the discretized problem we consider a $f_{(k)}$ that is constant along each $\mathcal{T}_{(k)}$, and consider $\mu_{(k, i, j)}$ and $\Phi_{(k, i, j)}$ that are constant along each $\mathcal{T}_{(k)} \times \Omega_{(i,j)}$. The vector $m_{(k, i, j)}$ has two components $m_{x, (k, i - 1/2, j)}$ and $m_{y, (k, i, j - 1/2)}$, that are constant along $\mathcal{T}_{(k)} \times \Omega_{(i - 1/2, j)}$ and $\mathcal{T}_{(k)} \times \Omega_{(i, j - 1/2)}$ respectively. Numerically $m$ quantifies the movement of density between each of the $\Omega_{(i,j)}$ and its spacial neighbors (i.e. $\Omega_{(i - 1, j)}, \Omega_{(i, j-1)}, \Omega_{(i + 1, j)}$, and $\Omega_{(i, j+1)}$) and so it is natural to define the components of $m$ not on $\Omega_{(i,j)}$ but rather on $\Omega_{(i - 1/2, j)}$, $\Omega_{(i + 1/2, j)}$, $\Omega_{(i, j - 1/2)}$ and $\Omega_{(i, j + 1/2)}$.

Using the above notation, we write the steps of the algorithm as:
\begin{align*}
    m_{x,(k, i - 1/2, j)} &= \begin{cases}
    \frac{\mu_{(k, i - 1, j)} + \mu_{(k, i - 1, j)}}{\mu_{(k, i, j)} + \mu_{(k, i - 1, j)} + 2\tau_1}\left(\tau_1 + \nabla_x\Phi_{(k, i - 1/2, j)} + m_{x,(k, i - 1/2, j)} \right) &\text{ if } i = 1,\dots, n_x - 1\\
    0 &\text{ if } i = 0, n_x
    \end{cases}\\
    m_{y,(k, i, j - 1/2)} &= \begin{cases}
    \frac{\mu_{(k, i, j)} + \mu_{(k, i, j - 1)}}{\mu_{(k, i, j)} + \mu_{(k, i, j - 1)} + 2\tau_1}\left(\tau_1 + \nabla_y \Phi_{(k, i, j - 1/2)} + m_{x,(k, i, j - 1/2)} \right) &\text{ if } j = 1,\dots,n_y - 1\\
    0 &\text{ if } j = 0, n_y
    \end{cases}\\
    u_{(k, i, j)} &= \text{root}^+(1, -(\tau_1 * \partial_t \Phi_{(k, i, j)} + u_{(k, i, j)}), 0,\\
    &\frac{-\tau_1}{8}\left ((m_{(k, i + 1/2, j)} + m_{(k, i - 1/2, j)})^2 + (m_{(k, i, j + 1/2)} + m_{(k, i, j - 1/2)})^2\right )) \\
    f_{(k)} &= \frac{\alpha}{\alpha + \tau_1}\left ( \tau_1 + \sum_{i}\sum_{j}\Phi_{(k,i,j)}\Delta x \Delta y+ f_{(k)} \right )\\
    \Phi_{(k, i, j)} &= \tau_2 * \left ( \partial_t \tilde u_{(k, i, j)} + \nabla \cdot  \tilde m_{(k, i, j)} - \tilde f_{(k)}\right ) + \Phi_{(k, i, j)}
\end{align*}
where
\begin{align*}
    \nabla_x \Phi_{(k, i - 1/2, j)} &= \frac{\Phi_{(k, i, j)} - \Phi_{(k, i - 1, j)}} {\Delta x} \\ 
    \nabla_y \Phi_{(k, i, j - 1/2)} &= \frac{\Phi_{(k, i, j)} - \Phi_{(k, i, j - 1)}} {\Delta y}; \\
    \partial_t \Phi_{(k, i, j)} &= \begin{cases}
    \frac{1}{\Delta t}\left (\frac{\Phi_{(1, i, j)}}{2} + \Phi_{(0, i, j)}\right ) \text{ if } k = 0\\
    \frac{1}{\Delta t} \left ( \frac{\Phi_{(2, i, j)}}{2} - \Phi_{0, i, j}\right ) \text{ if } k = 1\\
    \frac{1}{2\Delta t}\left ( \Phi_{(k + 1, i, j)} - \Phi_{(k - 1, i, j)}\right ) \text{ if } 1 < k < n_t - 2 \\
    \frac{1}{\Delta t} \left ( \Phi_{(n_t - 1, i, j)} - \frac{\Phi_{(n_t - 3, i, j)}}{2}\right ) \text{ if } k = n_t - 2\\
    \frac{1}{\Delta t} \left ( -\Phi_{(n_t - 1, i, j)} - \frac{\Phi_{(n_t - 2, i, j)}}{2}\right ) \text{ if } k = n_t - 1
    \end{cases}\\
    \text{root}^+(a,b,c,d) &= \text{ the largest real solution to } ax^3 + bx^2 + cx + d\\
    \partial_t u_{(k, i, j)} &= \begin{cases}
    \frac{1}{\Delta t}\left (u_{(1, i, j)} - u_{(0, i, j)}\right ) \text{ if } k = 0 \\
    \frac{1}{2 \Delta t} \left ( u_{(k + 1, i, j)} - u_{(k - 1, i, j)}\right ) \text{ if } 0 < k < n_t - 1 \\
    \frac{1}{\Delta t}\left (u_{(n_t - 1, i, j)} - u_{(n_t - 2, i, j)}\right ) \text{ if } k = n_t - 1
    \end{cases}\\
    \nabla \cdot m_{(k, i, j)} &= \frac{m_{x, (k, i + 1/2, j)} - m_{x, (k, i - 1/2, j)}}{\Delta x} + \frac{m_{y, (k, i, j + 1/2)} - m_{y, (k, i, j - 1/2)}}{\Delta y}.
\end{align*}
Note that the unusual boundary conditions of $\partial_t \Phi$ arise from the need to satisfy
\begin{equation*}
    \sum_k \Phi_{(k, i, j)}\partial_t u_{(k, i, j)} \Delta t = -\sum_k \partial_t \Phi_{(k, i, j)} u_{(k, i, j)}  \Delta t \quad \forall i, j.
\end{equation*}

\subsection{Numerical Experiments}

\begin{table}[ht!]
    \centering
    \begin{tabular}{l l| l l}
         Parameter & Value & Parameter & Value\\\hline\hline \multicolumn{2}{l}{\textbf{Discretization}}& \multicolumn{2}{l}{\textbf{Optimization}}\\ \hline
         $n_t$ & 15 & Iterations & 200,000 \\
         $n_x$ & 35 & $\tau_1$ & $10^{-3}$\\
         $n_y$ & 35 & $\tau_2$ & $10^{-1}$\\
          & & $\alpha$ & 100
    \end{tabular}
    \caption{Numerical parameters for our experiments. Note that for our one dimensional experiments, $n_y$ has no value.}
    \label{table:experimental-parameters}
\end{table}

Now we present our numerical results. The first two experiments are in one dimension, and the rest are in two. The numerical parameters for our experiments are given in Table \ref{table:experimental-parameters}.

\subsection{Experiment 1}
\begin{figure}[h!]
    \centering
    \begin{subfigure}{.31\linewidth}
        \centering
        \includegraphics[width=\textwidth,keepaspectratio]{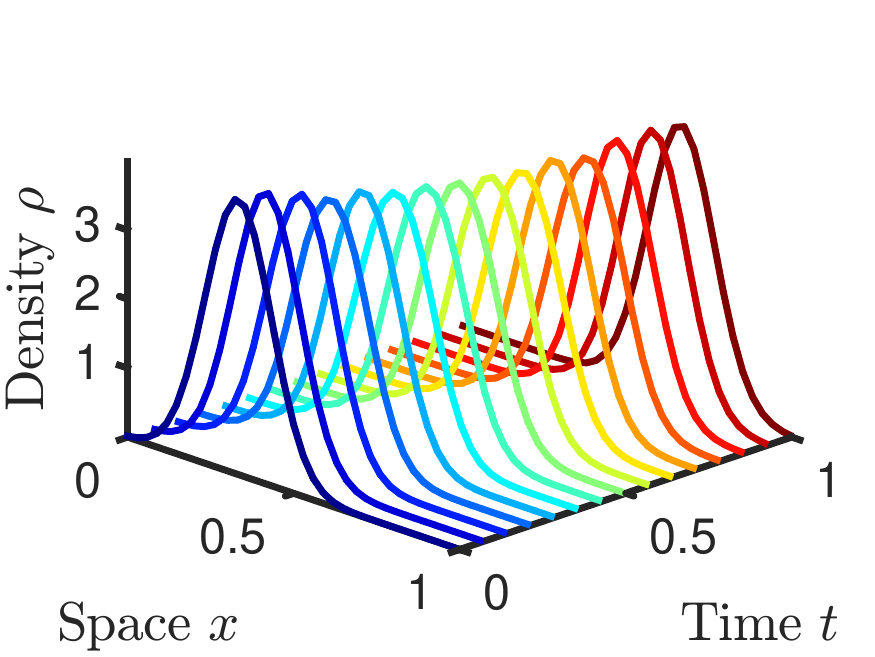}
        \caption{$u(t, x, y)$}
        \label{fig:ex-1:no-f}
    \end{subfigure}
    \begin{subfigure}{.31\linewidth}
        \centering
        \includegraphics[width=\textwidth,keepaspectratio]{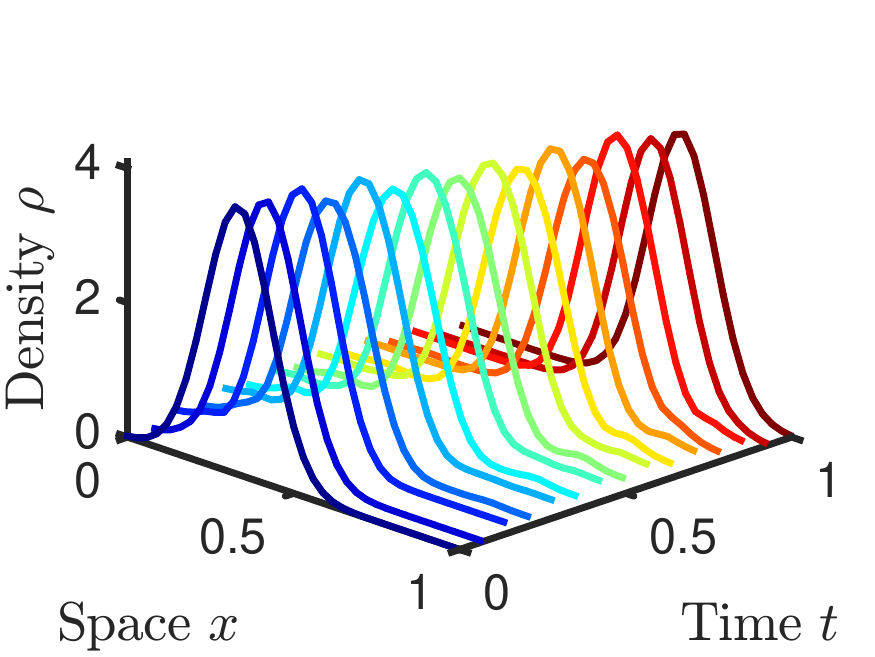}
        \caption{$u(t, x, y)$}
        \label{fig:ex-1:with-f}
    \end{subfigure}
    \begin{subfigure}{.31\linewidth}
        \centering
        \includegraphics[width=\textwidth,keepaspectratio]{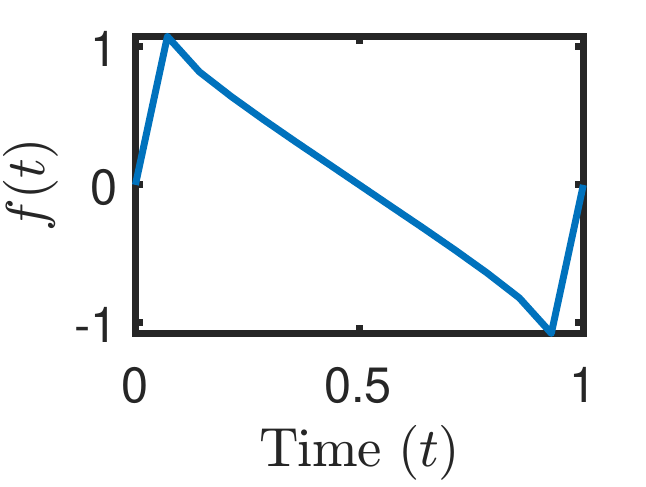}
        \caption{$f(t)$}
        \label{fig:ex-1:f}
    \end{subfigure}
    \caption{A plot of (A) $W_2(\rho_0, \rho_1)$, (B) $UW_2(\rho_0, \rho_1)$ and (C) $f(t)$ in the unbalanced case.}
    \label{fig:ex-1}
\end{figure}

Here we consider the problem where $\rho_0$ and $\rho_1$ are both one dimensional Gaussians of equal integral, $\Omega = [0,1]$ and
\begin{align*}
    \rho_0 &= N\left (x; \frac{1}{3}, 0.1\right)\\
    \rho_1 &= N\left (x; \frac{2}{3}, 0.1\right)\\
    N(x;\mu, \sigma^2) &= C e^{\frac{(x - \mu)^2}{2\sigma^2}} \text{ where $C$ is such that } \int_{\Omega} N(x; \mu, \sigma^2) dx = 1
\end{align*}
where $\sigma_0 = \frac{1}{3}, \sigma_1 = \frac{2}{3}, \mu_0 = \mu_1 = 0.1$. We plot the results in Figure \ref{fig:ex-1}. In this case the input densities are balanced and so $W_2(\rho_0, \rho_1)$ and $UW_2(\rho_0, \rho_1)$ appear similar. Indeed $UW_2(\rho_0, \rho_1) = 0.055$ and $W_2(\rho_0, \rho_1) = 0.056$.

Note that even in this simple case the behavior of $f(t)$ is nuanced. In this case, $\rho_0$ and $\rho_1$ are smooth, of equal integral and $W_2(\rho_0,\rho_1)$ is given by a simple analytical formula, and $f(t)$ is not identically zero. Integrating Equation \ref{UC} in space and time yields $ |\Omega| \int_{[0,1]} f(t) dt = \int_\Omega \rho_1 dx - \int_{\Omega} \rho_0 dx$, and so for balanced inputs $\int_{[0,1]} f(t) dt = 0$, but experiment 1 shows that $f \not \equiv 0$.

\subsection{Experiment 2}

\begin{figure}[h!]
    \centering
    \begin{subfigure}{.31\linewidth}
        \centering
        \includegraphics[width=\textwidth,keepaspectratio]{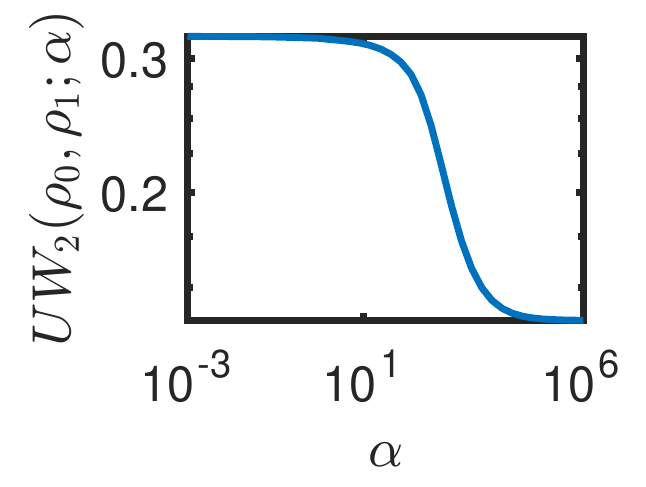}
        \caption{$UW_2(\rho_0', \rho_1; \alpha)$}
        \label{fig:ex-2-balanced:cost}
    \end{subfigure}
    \begin{subfigure}{.31\linewidth}
        \centering
        \includegraphics[width=\textwidth,keepaspectratio]{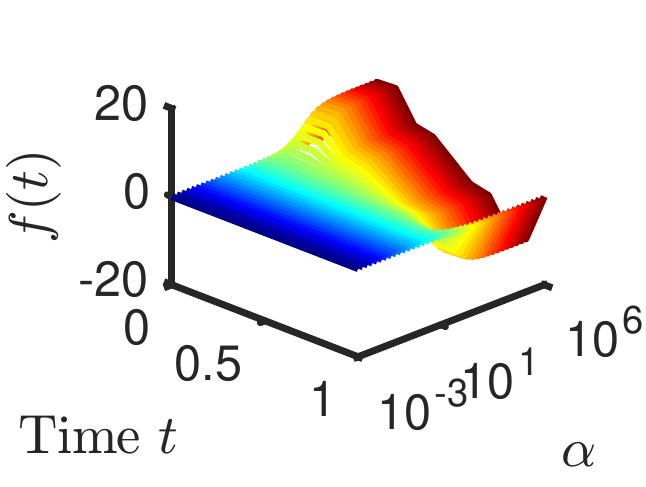}
        \caption{$f'(t; \alpha)$}
        \label{fig:ex-2-balanced:f}
    \end{subfigure}
    \begin{subfigure}{.31\linewidth}
        \centering
        \includegraphics[width=\textwidth,keepaspectratio]{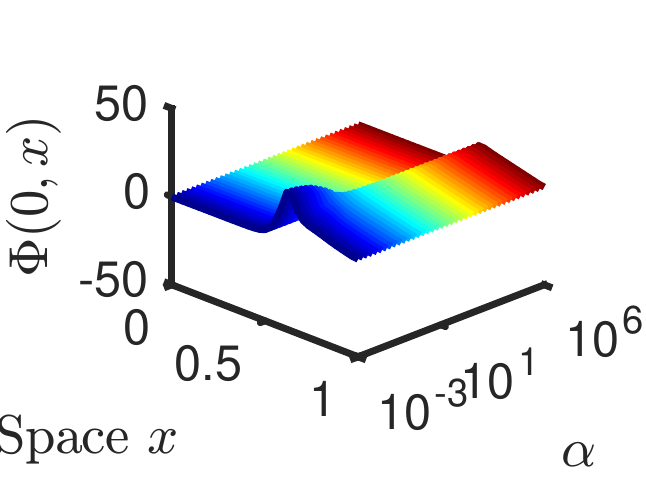}
        \caption{$\Phi'(t, x; \alpha)$}
        \label{fig:ex-2-balanced:phi}
    \end{subfigure}
    \\
    \begin{subfigure}{.31\linewidth}
        \centering
        \includegraphics[width=\textwidth,keepaspectratio]{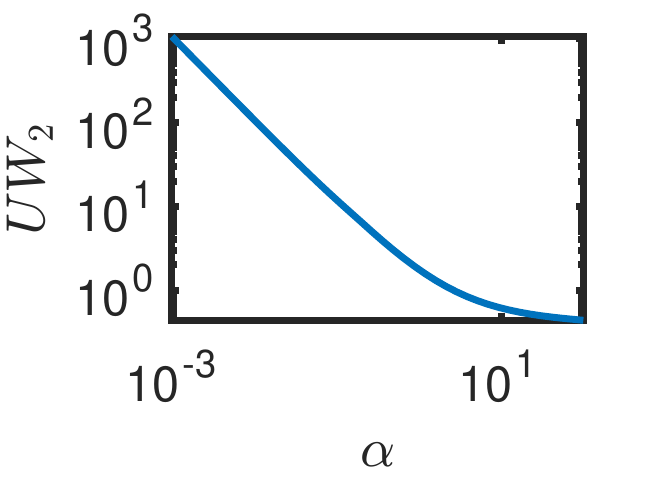}
        \caption{$UW_2(\rho_0, \rho_1; \alpha)$}
        \label{fig:ex-2-unbalanced:cost}
    \end{subfigure}
    \begin{subfigure}{.31\linewidth}
        \centering
        \includegraphics[width=\textwidth,keepaspectratio]{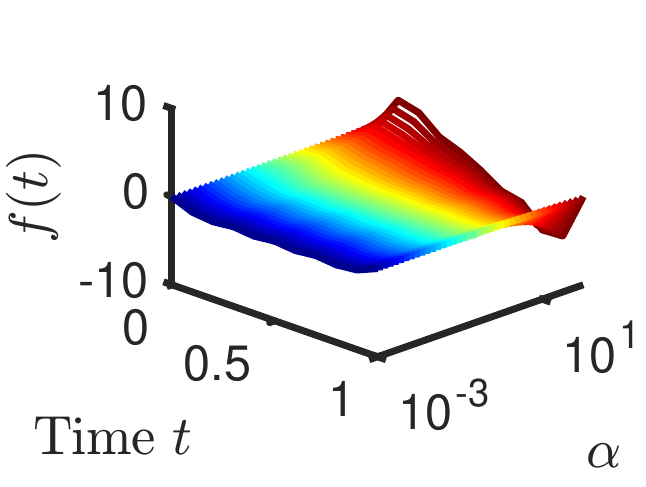}
        \caption{$f(t; \alpha)$}
        \label{fig:ex-2-unbalanced:f}
    \end{subfigure}
    \begin{subfigure}{.31\linewidth}
        \centering
        \includegraphics[width=\textwidth,keepaspectratio]{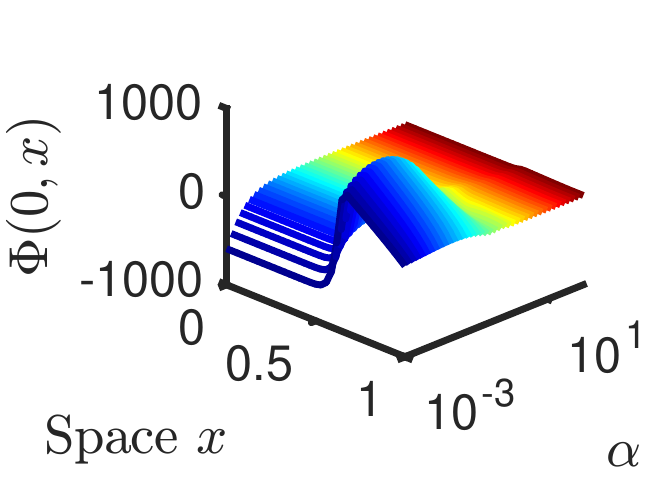}
        \caption{$\Phi(t, x; \alpha)$}
        \label{fig:ex-2-unbalanced:phi}
    \end{subfigure}
    \caption{A plot of the asymptotic behavior of $UW_2$ in $\alpha$ with balanced and unbalanced inputs. Balanced: (A) $UW_2(\rho_0', \rho_1; \alpha)$, (B) $f'(t; \alpha)$, (C) $\Phi'(t, x; \alpha)$, and unbalanced: (D) $UW_2(\rho_0, \rho_1; \alpha)$, (E) $f(t; \alpha)$, (F) $\Phi(t, x; \alpha)$.}
    \label{fig:ex-2}
\end{figure}

Again consider $\Omega = [0,1]$, however in this experiment we analyse the asymptotic behavior of $UW_2(\rho_0, \rho_1)$ as a function of $\alpha$ and $\alpha \rightarrow 0$ and $\alpha \rightarrow \infty$. Here 
\begin{align*}
    \rho_0 &= N\left (x; 0, 0.1\right) + N\left (x; \frac{1}{3}, 0.1\right)\\
    \rho_0' &= \frac{1}{2} \left (N\left (x; 0, 0.1\right) + N\left (x; \frac{1}{3}, 0.1\right)\right)\\
    \rho_1 &= N\left (x; \frac{2}{3}, 0.1\right).
\end{align*}
The balanced case refers to $UW_2(\rho_0', \rho_1)$, and the unbalanced refers to $UW_2(\rho_0, \rho_1)$. In both cases we compute the unnormalized Wasserstein distance. The results are given in Figure \ref{fig:ex-2}.

Figures \ref{fig:ex-2-balanced:cost} - \ref{fig:ex-2-balanced:phi} show that (at least numerically) $UW_2(\rho_0, \rho_1 ; \alpha), f(t,\alpha)$ and $\Phi(t,x;\alpha)$ converge as $\alpha \rightarrow 0^+$, $\alpha \rightarrow \infty$ when $\int_{\Omega}\rho_0 dx = \int_{\Omega}\rho_1 dx$. Further is seems plausible that for balanced inputs $UW_2(\rho_0, \rho_1 ; \alpha) \rightarrow W_2(\rho_0, \rho_1)$ as $\alpha \rightarrow 0^+$. For any $\alpha$ the $u, m$ and $\Phi$ from $W_2(\rho_0, \rho_1)$ along with $f(t) \equiv 0$ satisfy the constraint of Equation \ref{UC}. Formally sending $\alpha \rightarrow \infty$ causes $f(t)$ to 0.

Figures \ref{fig:ex-2-unbalanced:cost} - \ref{fig:ex-2-unbalanced:phi} illustrate the asymptotic behavior of $UW_2(\rho_0, \rho_1; \alpha)$ w.r.t. $\alpha$ when the inputs are unbalanced. In that case we (numerically) see that as $\alpha \rightarrow 0$, $f(t;\alpha)$ converges to a non-zero value, and both $UW_2(\rho_0, \rho_1;\alpha)$ and $\Phi(t,x;\alpha)$ diverge. This too is consistent with the formal argument that $UW_2(\rho_0, \rho_1 ; \alpha) \rightarrow W_2(\rho_0, \rho_1)$ as $\alpha \rightarrow 0^+$.

In a predecessor of this work \cite{bb} the authors solve for $W_2(\rho_1, \rho_2)$ using Lagrange multipliers in a similar formulation to equations \eqref{UOT}, \eqref{UC}. In their work the Lagrange multiplier $\Phi(t,x)$ is given up to an additive constant. If indeed $UW_2(\rho_0, \rho_1 ; \alpha) \rightarrow W_2(\rho_0, \rho_1)$ as $\alpha \rightarrow 0^+$ and $\Phi(t,x;\alpha)$ does converge then $\Phi(t,x;0^+)$ is given uniquely (as a limit) and there is no issue of undetermined constants.

\subsection{Experiment 3}

\begin{figure}[ht!]
    \centering
    \begin{subfigure}{.31\linewidth}
        \centering
        \includegraphics[width=\textwidth,keepaspectratio]{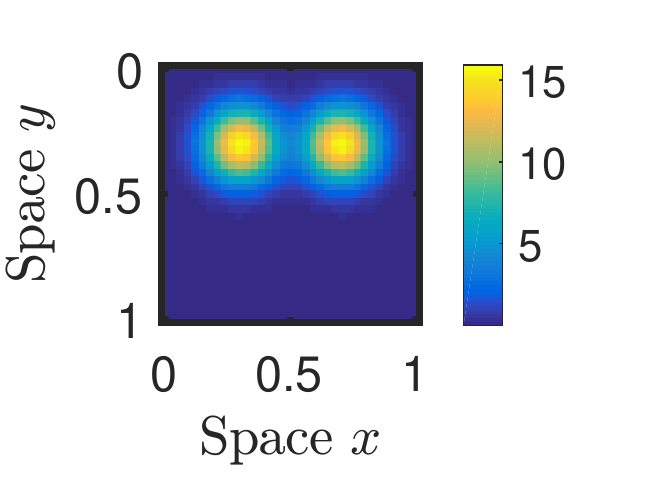}
        \caption{$\mu(0.00,x,y)$}
        \label{fig:ex-3:mu-0.00}
    \end{subfigure}
    \begin{subfigure}{.31\linewidth}
        \centering
        \includegraphics[width=\textwidth,keepaspectratio]{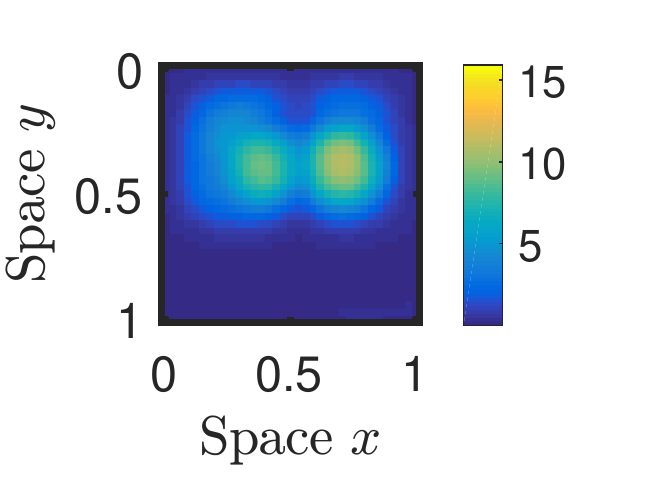}
        \caption{$\mu(0.21,x,y)$}
        \label{fig:ex-3:mu-0.21}
    \end{subfigure}
    \begin{subfigure}{.31\linewidth}
        \centering
        \includegraphics[width=\textwidth,keepaspectratio]{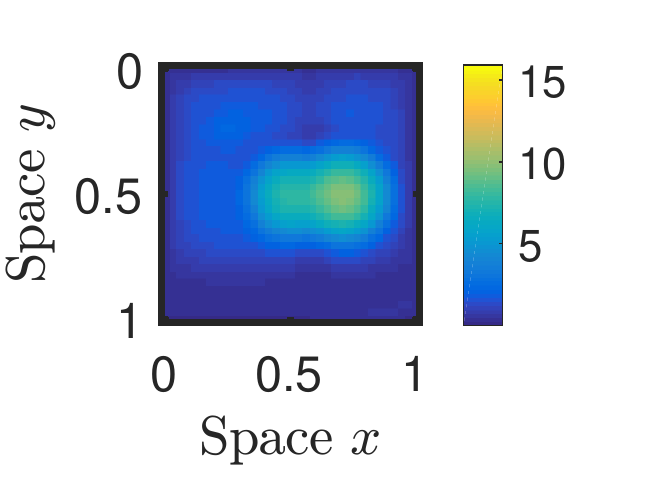}
        \caption{$\mu(0.50,x,y)$}
        \label{fig:ex-3:mu-0.50}
    \end{subfigure}
    \\
    \begin{subfigure}{.31\linewidth}
        \centering
        \includegraphics[width=\textwidth,keepaspectratio]{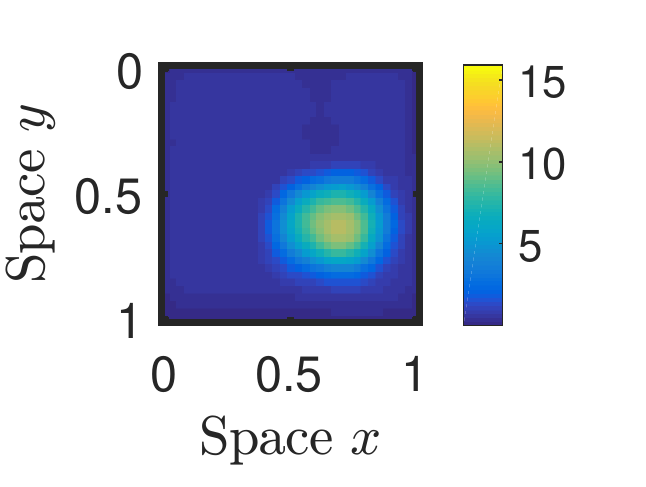}
        \caption{$\mu(0.79,x,y)$}
        \label{fig:ex-3:mu-0.79}
    \end{subfigure}
    \begin{subfigure}{.31\linewidth}
        \centering
        \includegraphics[width=\textwidth,keepaspectratio]{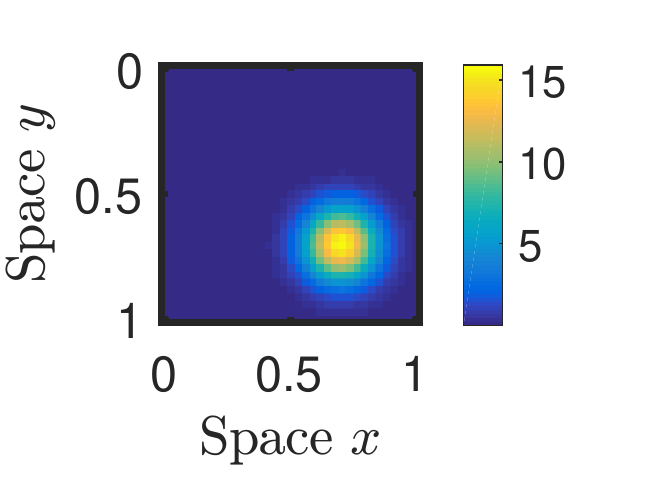}
        \caption{$\mu(1.00,x,y)$}
        \label{fig:ex-3:mu-1.00}
    \end{subfigure}
    \begin{subfigure}{.31\linewidth}
        \centering
        \includegraphics[width=\textwidth,keepaspectratio]{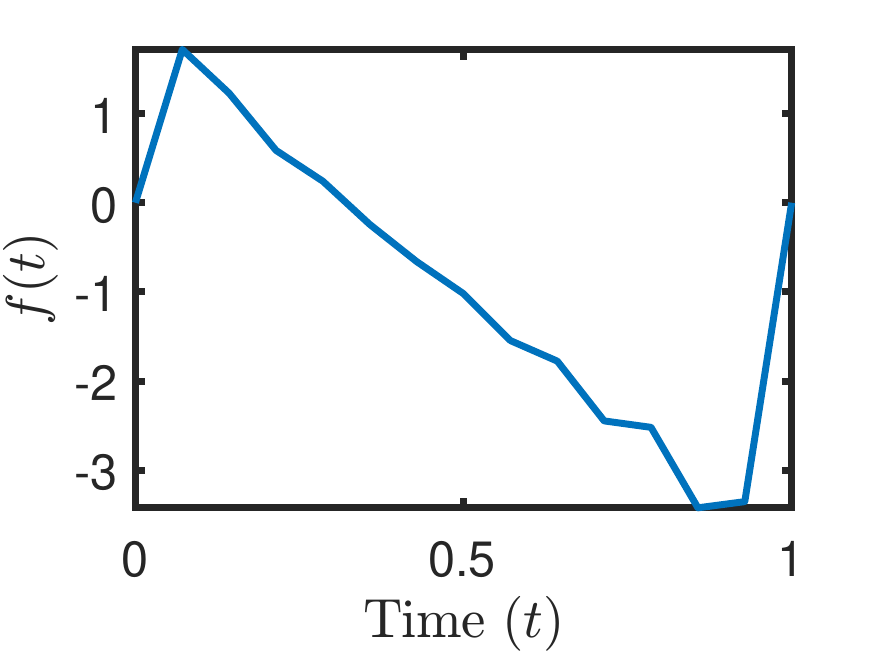}
        \caption{$f(t)$}
        \label{fig:ex-3:f}
    \end{subfigure}
    \caption{Plots of the $u(t,x,y)$ and $f(t)$ for $UW_2(\rho_0, \rho_1)$. (A) $\mu(0.00, x, y)$, (B) $\mu(0.21, x, y)$, (C) $\mu(0.50, x, y)$, (D) $\mu(0.79, x, y)$, (E) $\mu(1.00, x, y)$, (F) $f(t)$.}
    \label{fig:ex-3}
\end{figure}

Now consider the two dimensional problem where $\Omega = [0, 1]^2$. In this case
\begin{align*}
    \rho_0(x,y) &= N\left (x, y;0.3, 0.3, 0.1, 0.1 \right ) + N\left (x, y;0.7, 0.3, 0.1, 0.1 \right )\\
    \rho_1(x,y) &= N\left (x, y;0.7, 0.7, 0.1, 0.1 \right )\\
    N(x, y;\mu_1, \mu_2, \sigma_1^2, \sigma_2^2) &= C e^{\frac{(x - \mu_1)^2}{2\sigma_1^2} + \frac{(y - \mu_2)^2}{2\sigma_2^2}},
\end{align*}
where $C$ is a normalization constant such that $\int_{\Omega} N(x, y;\mu_1, \mu_2, \sigma_1^2, \sigma_2^2) dxdy = 1$. The results from our experiments are shown in Figure \ref{fig:ex-3}. Note that although the mass of $\rho_0$ is twice that of $\rho_1$, the optimal $f(t)$ is not non-positive. Indeed from $t = 0$ to $t \approx \frac{1}{4}$, $f(t)$ is positive, before staying non-positive for the rest of the interval. This again illustrates that even in the case of gaussian movement the behavior of $f(t)$ is nuanced, and violates naive basic intuition.

\subsection{Experiment 4}

\begin{figure}[h!]
    \centering
    \begin{subfigure}{.31\linewidth}
        \centering
        \includegraphics[width=\textwidth,keepaspectratio]{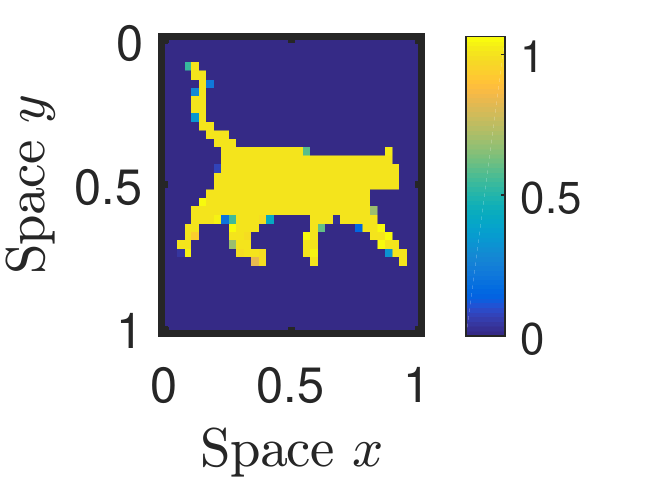}
        \caption{$\mu(0.00,x,y)$}
        \label{fig:ex-4:mu-0.00}
    \end{subfigure}
    \begin{subfigure}{.31\linewidth}
        \centering
        \includegraphics[width=\textwidth,keepaspectratio]{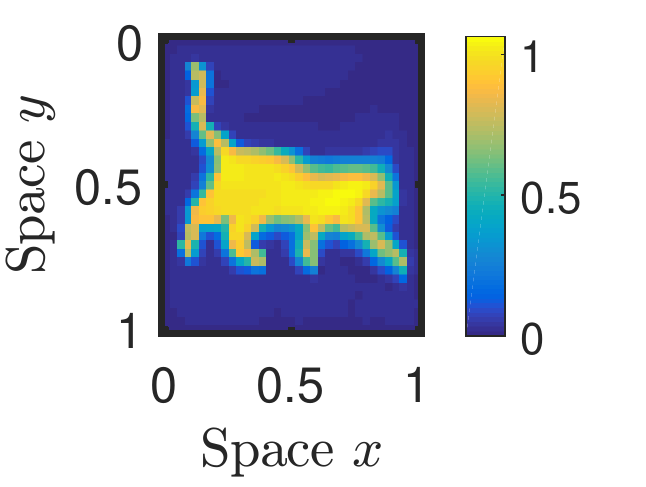}
        \caption{$\mu(0.21,x,y)$}
        \label{fig:ex-4:mu-0.21}
    \end{subfigure}
    \begin{subfigure}{.31\linewidth}
        \centering
        \includegraphics[width=\textwidth,keepaspectratio]{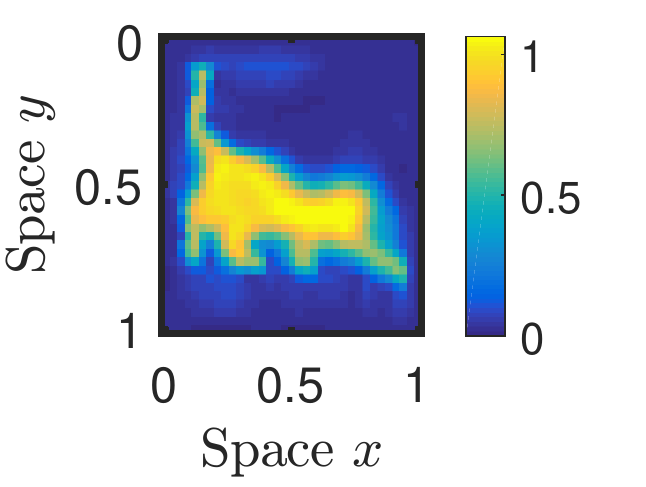}
        \caption{$\mu(0.50,x,y)$}
        \label{fig:ex-4:mu-0.50}
    \end{subfigure}
    \\
    \begin{subfigure}{.31\linewidth}
        \centering
        \includegraphics[width=\textwidth,keepaspectratio]{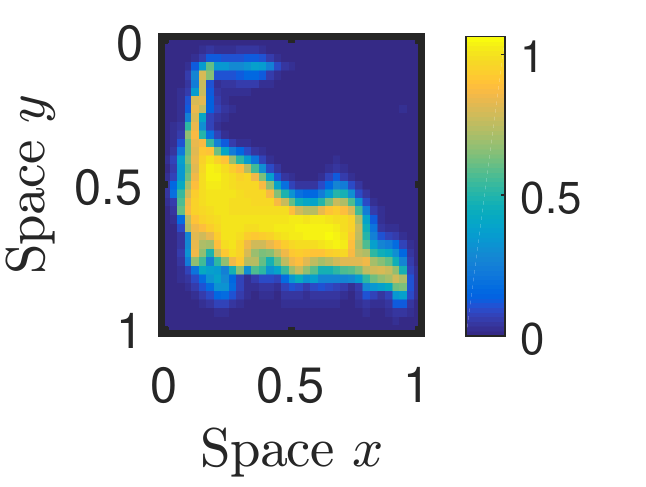}
        \caption{$\mu(0.79,x,y)$}
        \label{fig:ex-4:mu-0.79}
    \end{subfigure}
    \begin{subfigure}{.31\linewidth}
        \centering
        \includegraphics[width=\textwidth,keepaspectratio]{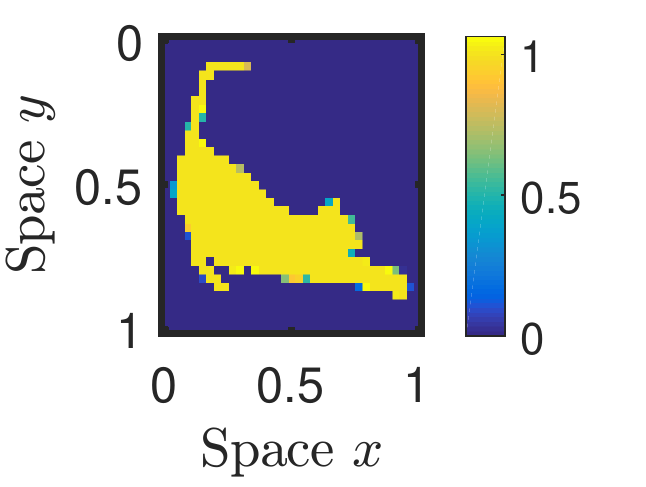}
        \caption{$\mu(1.00,x,y)$}
        \label{fig:ex-4:mu-1.00}
    \end{subfigure}
    \begin{subfigure}{.31\linewidth}
        \centering
        \includegraphics[width=\textwidth,keepaspectratio]{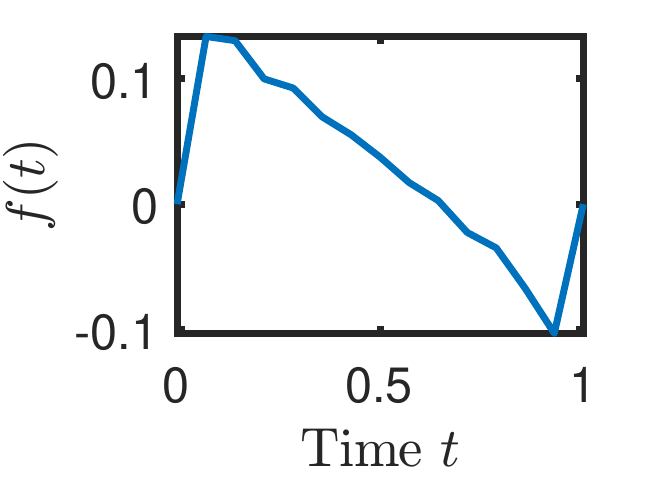}
        \caption{$f(t)$}
        \label{fig:ex-4:f}
    \end{subfigure}
    \caption{Plots of the $u(t,x,y)$ and $f(t)$ for $UW_2(\rho_0, \rho_1)$. (A) $\mu(0.00, x, y)$, (B) $\mu(0.21, x, y)$, (C) $\mu(0.50, x, y)$, (D) $\mu(0.79, x, y)$, (E) $\mu(1.00, x, y)$, (F) $f(t)$.}
    \label{fig:ex-4}
\end{figure}

Consider again the two dimensional problem, however this time we choose $\rho_0$ and $\rho_1$ to be the cats in \cite{LiRyuOsherYinGangbo2018_parallel}. Our results are summarized in Figure \ref{fig:ex-4}. This illustrates that our new method can be used as a general purpose OT solver for unbalanced inputs, and so can be used to interpolate between two functions.

\subsection{$L^1$ unnormalized Wasserstein metric}

\begin{figure}[h!]
    \centering
    \begin{subfigure}{.31\linewidth}
        \centering
        \includegraphics[width=\textwidth,keepaspectratio]{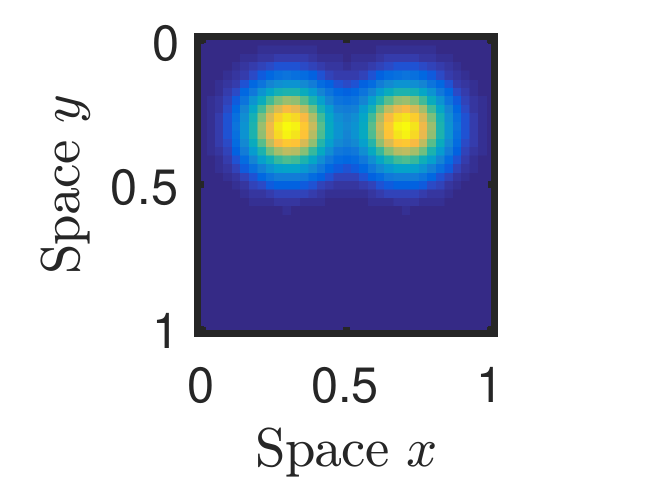}
        \caption{$\rho_0$}
        \label{fig:ex-5:gaussians:rho-0}
    \end{subfigure}
    \begin{subfigure}{.31\linewidth}
        \centering
        \includegraphics[width=\textwidth,keepaspectratio]{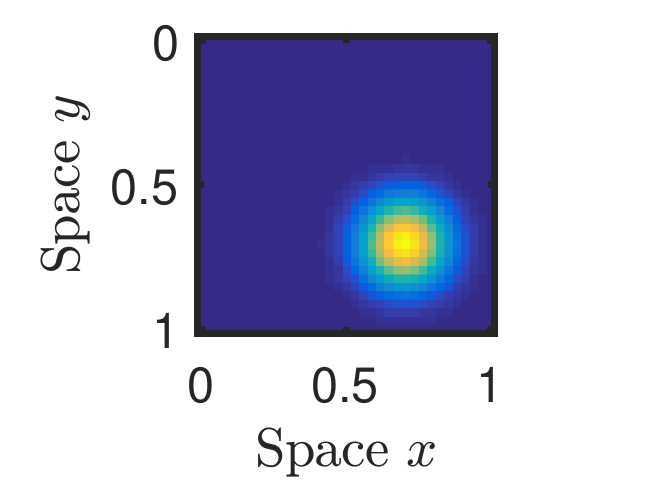}
        \caption{$\rho_1$}
        \label{fig:ex-5:gaussians:rho-1}
    \end{subfigure}
    \begin{subfigure}{.31\linewidth}
        \centering
        \includegraphics[width=\textwidth,keepaspectratio]{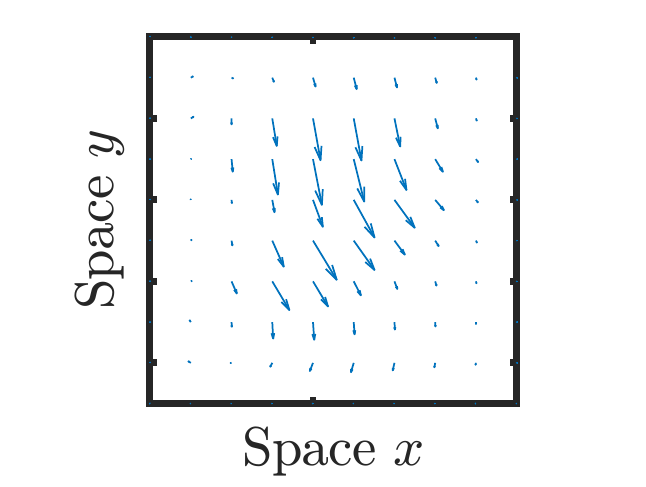}
        \caption{$m(x)$}
        \label{fig:ex-5:gaussians:flux}
    \end{subfigure}
    \\
    \begin{subfigure}{.31\linewidth}
        \centering
        \includegraphics[width=\textwidth,keepaspectratio]{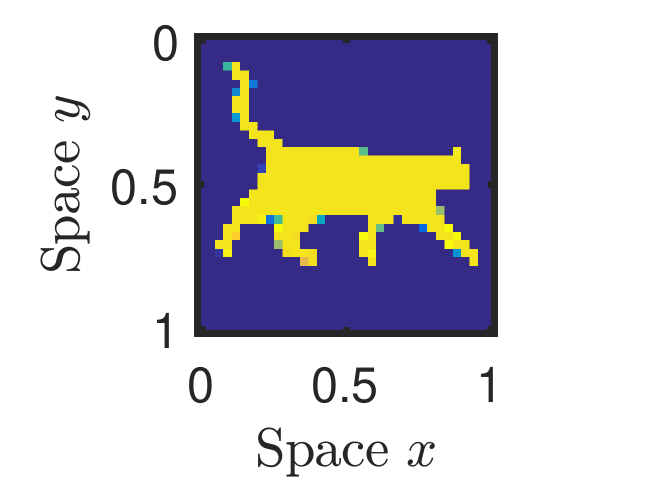}
        \caption{$\rho_0$}
        \label{fig:ex-5:cats:rho-0}
    \end{subfigure}
    \begin{subfigure}{.31\linewidth}
        \centering
        \includegraphics[width=\textwidth,keepaspectratio]{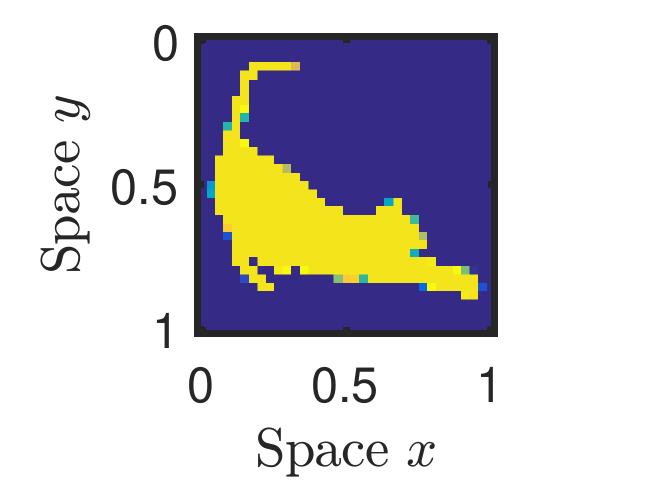}
        \caption{$\rho_1$}
        \label{fig:ex-5:cats:rho-1}
    \end{subfigure}
    \begin{subfigure}{.31\linewidth}
        \centering
        \includegraphics[width=\textwidth,keepaspectratio]{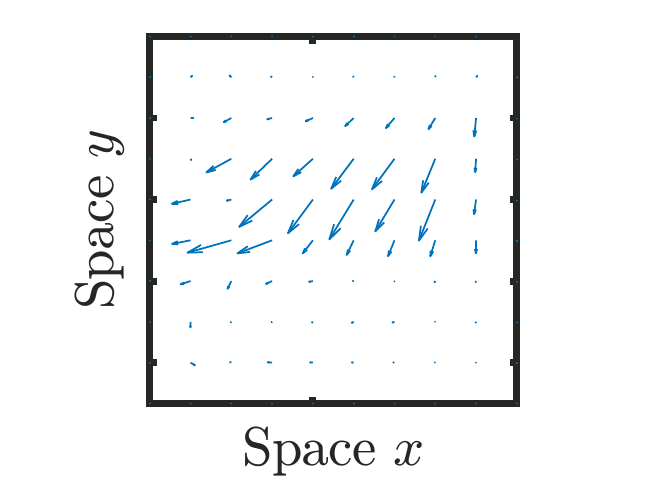}
        \caption{$m(x)$}
        \label{fig:ex-5:cats:flux}
    \end{subfigure}
    \caption{Plots of the $\rho_0$, $\rho_1$ and $m(x)$ for $UW_1(\rho_0, \rho_1)$ for the two gaussian movement (A) $\rho_0$, (B) $\rho_1$, (C) $m(x)$ and two (D) $\rho_0$, (E) $\rho_1$, (F) $m(x)$.}
    \label{fig:ex-5}
\end{figure}

In this subsection, we also present several numerical results for $UW_1$ in Figure \ref{fig:ex-5}. In \cite{PO} the authors develop the $UW_1$ metric (called the $\struc{\cdot}$ in that work) and show that it has the desirable property that is insensitive to noise and sensitive to the underlying structure. Numerically $UW_1(\rho_0, \rho_1)$ is much easier to compute as the time dimension can be integrated out, so that $f$ is constant, and $\mu$, $m$ and $\Phi$ have no time-varying component.

\section{Discussion}
In this paper, we propose and solve an unnormalized optimal transport problem. We show that the proposed distance is well defined, and we obtain the minimizer using the same key Hamilton-Jacobi equation \eqref{HJBM}. More importantly, computing the $L^p$ unnormalized Wasserstein metric has essentially the same computational complexity as the normalized one. In the future, we intend to study these related geometric properties and applications in inverse problems, machine learning and mean field games. 
\bibliographystyle{abbrv}

\end{document}